\documentclass{article}
\usepackage{graphicx}

\usepackage[fleqn]{amsmath}

\usepackage{amsthm}
\usepackage{amssymb}
\usepackage{amsbsy}
\usepackage{amsfonts}
\usepackage{mathrsfs}
\usepackage{amsmath}
\usepackage[all]{xy}
\usepackage{amstext}
\usepackage{amscd}
\usepackage[dvips]{epsfig}
\usepackage{psfrag}
\usepackage{enumerate}
\usepackage{flafter}
\allowdisplaybreaks

\theoremstyle{plain}
\newtheorem{thm}{Theorem}[section]
\newtheorem{prop}[thm]{Proposition}
\newtheorem{cor}[thm]{Corollary}
\newtheorem{lem}[thm]{Lemma}
\theoremstyle{definition}
\newtheorem{exa}[thm]{Example}

\newtheorem{rem}[thm]{Remark}
\newtheorem{defn}[thm]{Definition}

\def\det{\mathop{\mathrm{det}}\nolimits}

\def\Hom{\mathop{\mathrm{Hom}}\nolimits}

\def\F{\mathop{\mathbb{F}}\nolimits}
\def\Fq{\mathop{\mathbb{F}_q}\nolimits}

\newcommand{\lra}{\longrightarrow}
\newcommand{\ra}{\rightarrow}
\newcommand{\N}{{\Bbb N}}
\newcommand{\Q}{{\Bbb Q}}

\newcommand{\Z}{{\Bbb Z}}
\newcommand{\C}{{\Bbb C}}

\newcommand{\SL}{{\rm SL}}
\newcommand{\GL}{{\rm GL}}
\newcommand{\Ker}{{\rm Ker}}

\newcommand{\Tr}{{\mathrm{Tr}}}
\newcommand{\Ind}{{\mathrm{Ind}}}

\usepackage{color}

\begin{document}
\large

\begin{center}{\bf \Large The Dijkgraaf-Witten invariants from second Chern classes} \end{center}
\begin{center}{ Takefumi Nosaka
\footnote{
E-mail address: {\tt nosaka@math.titech.ac.jp}
}}\end{center}

\begin{abstract}\baselineskip=12pt \noindent
Given a 3-cocycle $\psi$ in the cohomology of a finite group $G$, we can define the Dijkgraaf-Witten invariant of closed 3-manifolds.
In this paper, we focus on the case where $\psi$ is a 3-cocycle systematically transferred from the second Chern class of a complex representation of $G$,
and show some procedures for computing the invariant, and clarify its topological interpretation under some conditions.
\end{abstract}
\begin{center}
\normalsize
\baselineskip=11pt
{\bf Keywords} \\
\ \ \ Group cohomology, 3-manifold, Chern classes, splitting principle \ \ \
\end{center}
\begin{center}
\normalsize
\baselineskip=11pt
{\bf Subject Codes } \\
20J06, \ 57K31, \ 19L10,\ 58J28, \ \ \
\end{center}

\large
\baselineskip=16pt
\section{Introduction}\label{S00}

Let $M$ be a connected oriented closed 3-manifold with 
orientation 3-class $[M] \in H_3(M;\Z) \cong \Z$.
Let $G$ be a group, and $BG$ be an Eilenberg-MacLane space of type $(G,1)$. If $G=\pi_1(X)$, we can choose a classifying map $\iota : M \hookrightarrow B\pi_1(M)$, uniquely up to homotopy. 
Fix a group 3-cocycle $\psi \in H^3(BG; A)=H^3(G;A)$ with trivial coefficients. 
Then, for a group homomorphism $f: \pi_1(M) \ra G$, we can consider the pairing $\langle f^* (\psi),\iota_* [M] \rangle \in A$.
As an important example, if $G$ is a Lie group with discrete topology and $\psi$ is a Chern-Simons 3-class, the pairing
is called {\it the Chern-Simons invariant} \cite{CS} (see \S \ref{SS8884} for the definition).
Meanwhile, as another example, the Dijkgraaf-Witten invariant \cite[\S 6]{DW} is a toy model of the Chern-Simons invariant when $G$ is of finite order (See Appendix \ref{SS8884} for details).
Precisely, the invariant is defined as the following formal sum in the group ring $\Z[A]$:
\begin{equation}\label{seq75559}
\mathrm{DW}_{\psi }(M):= \sum_{f \in \Hom( \pi_1(M) , G )} 1_{\Z} \langle f^* (\psi), \iota_* [M] \rangle \in \Z[A]. \end{equation}
Here, $1_{\Z}$ is the unit of $\Z.$
The DW invariant seems simple; however, there are few examples of the resulting computations and studies of the invariant, unless $G$ is abelian or nilpotent.

To study such pairings and the invariant \eqref{seq75559} it is a problem to examine what kind of $\psi$ we should choose, 
and how we find an explicit expression of $\psi$ when $G$ is not abelian. 
Since the Chern-Simons 3-class is, roughly speaking, considered to be an inverse of a transgression of the second Chern classes of principle $\mathrm{GL}_n(\C)$-bundles (see, e.g., \cite{CS} or \cite[\S\S 3--5]{DW}), it is sensible to focus on a class of group 3-cocycles, which are like such inverses from second Chern classes of $G$ with $|G| < \infty$.
For this, if $|G| < \infty,$ notice the following isomorphisms obtained from the cohomology long exact sequence and the universal coefficient theorem: 
\begin{equation}\label{seq4559} H^{n+1}( G;\Z ) \stackrel{\beta^{-1}}{\cong} H^n(G;\Q/\Z) \cong H_n(G;\Z) , \end{equation}
where $n\geq 1$ and the first isomorphism is the inverse of the Bockstein map induced from
\begin{equation}\label{seq79}
0 \lra \Z \xrightarrow{\ \textrm{inclusion} \ } \Q \lra \Q/ \Z \lra 0 \quad \mathrm{(exact)}. \end{equation}
In addition, as seen in \cite{Ev2,Th}, any complex representation $\rho: G \ra \mathrm{GL}_n( \C)$ uniquely admits the Chern class $c_{i}(\rho) \in H^{2i}( G;\Z ) $.
Thus, we have a cohomology 3-class of the form
\begin{equation}\label{lklk}\beta^{-1}( c_2(\rho)) \in H^3(G;\Q/ \Z). \end{equation}%

In this paper, as a finite-order model of the Chern-Simons 3-class, we consider the Dijkgraaf-Witten invariant \eqref{seq75559} with $\psi = 12 \beta^{-1}( c_2(\rho)) $.
First, as a result of a Riemann-Roch type theorem, we develop a formula for describing $ 12 \beta^{-1}( c_2(\rho)) $ (see Proposition \ref{prop4784}),
and show (Section \ref{SS3}) that, under a certain condition, the twelvefold invariant 
can be computed from the cohomology rings of some covering spaces of $M$.
For example, the condition holds if $G= \SL_2(\Fq)$ with $q>10$; see Corollary \ref{ex373413} (see also Remark \ref{ex34244} for other groups).
In Section \ref{SS8884}, from a viewpoint of the Chern-Simons invariant, we suggest a procedure for computing the Dijkgraaf-Witten invariants of some Seifert 3-manifolds. 
Finally, in Section \ref{SS4} we observe other approaches to the invariants with $\psi = \beta^{-1}( c_2(\rho)) $. 

\

\noindent
{\bf Conventional notation.} \ Let $G$ be a finite group, and $M$ be a connected closed 3-manifold with orientation class $[M] \in H_3(M;\Z) \cong \Z$.
By $BG$ we mean an Eilenberg-MacLane space of type $(G,1)$.




{


\section{Twelvefold second Chern classes}\label{SS2}
Let $\rho : G \ra \mathrm{GL}_n(\C )$ be any complex representation,
and $c_i(\rho) \in H^{2i}(G ;\Z)$ be the Chern class of $\rho$; see, e.g., \cite{Ev2,Th} for the definition.
The purpose of this section is to develop an algorithm to describe the twelvefold 3-class $12 \beta^{-1}(c_2(\rho)) .$



Let $c(\rho) \in H^*(G;\Z)$ be the total Chern class, that is, $c(\rho) = 1 +c_1(\rho)+ c_2(\rho) + \cdots + c_n(\rho).$
Let $R_{\mathbb{C}}(G)$ be the complex representation ring of $G$.
As is known as Brauer theorem, there are finitely many $n_i (\rho) \in \Z$ and subgroups $H_i \subset G$ with one--dimensional representation
$\phi_i:H_i \ra \mathrm{GL}_1(\C)$ such that
\begin{equation}\label{seq7712} \rho= \sum_{i} n_i (\rho)\mathrm{Ind}_{H_i}^G( \phi_i) \in R_{\C}(G).\end{equation}
Here, $\mathrm{Ind}_{H_i}^G( \phi_i) $ is the induced representation.
By the Whitney sum formula of $c(\rho)$, we have
\begin{equation}\label{seq771} c(\rho) = \prod (1 +c_1 ( \mathrm{Ind}_{H_i}^G( \phi_i)) + c_2( \mathrm{Ind}_{H_i}^G( \phi_i)) + \cdots )^{n_i(\rho)}\in H^*( G;\Z) \end{equation}
To summarize, to compute $c_2(\rho)$, we may focus only on $c_1 ( \mathrm{Ind}_{H_i}^G( \phi_i)) $ and $c_2( \mathrm{Ind}_{H_i}^G( \phi_i)) $.

Thus, we shall restrict ourselves to a subgroup $H \subset G$ and a one--dimensional representation
$\phi: H \ra \mathrm{GL}_1(\C) =\C^{\times}$, and show Lemma \ref{prop1} below.
Identify $\Q/\Z$ with the multiplicative subgroup $\{ \exp (2 \pi \sqrt{-1} r ) \mid r \in \Q \} \subset \C^{\times}$.
Since $|H| < \infty $, the image of $\phi$ is contained in $\Q/\Z$; we may regard $\phi$ as a 1-cocycle of $H$ with trivial coefficients.
Denote the transfer map by $\mathrm{Tr}^G_H: H^n (H;A) \ra H^n (G;A) $; see, e.g., \cite[\S III.9]{Br} for the definition.
Then, the following claims that the cohomology 3-class $12 \beta^{-1} \bigl( c_2( \mathrm{Ind}_{H}^G( \phi) )\bigr)$ can be computed from $\beta$ and the transfers.
\begin{lem}\label{prop1}
Take another subgroup $H' \subset G$ with a representation $\phi' : H' \ra \mathrm{GL}_1(\C)$.
Then, the following two identities hold:
\begin{equation}\label{s256}
2\beta^{-1} \bigl( c_1( \mathrm{Ind}_{H}^G( \phi) ) \smile c_1( \mathrm{Ind}_{H ' }^G(\phi' ) ) \bigr) = 2 \mathrm{Tr}^G_{H}( \phi ) \smile_{\Q/\Z} \beta \circ
\mathrm{Tr}^G_{H'}( \phi ' ) ,
\end{equation}
\begin{equation}\label{s23}
12 \beta^{-1} \bigl( c_2( \mathrm{Ind}_{H}^G( \phi) )\bigr) =
6\bigl( \mathrm{Tr}^G_{H}( \phi ) \smile_{\Q/\Z} \beta \circ \mathrm{Tr}^G_H( \phi ) - \mathrm{Tr}^G_H ( \phi \smile \beta ( \phi) ) \bigr) \in H^3(G;\Q/\Z). \end{equation}
Here, $\smile_{\Q/\Z}$ means the cup product with coefficients induced from the biadditive map $\Z \times \Q/\Z \ra \Q/\Z$ that sends $(n,[a])$ to $[an]$.
Furthermore, if $\phi$ extends to a 1-cocycle $G \ra \Q/\Z $, then 
\eqref{s23} vanishes.
\end{lem}
\begin{proof}
Recall the integral Riemann-Roch theorem in group cohomology (see, e.g., \cite[Theorem 6.3]{Th}), which 
partially\footnote{We here explain the partial deduction in details: In \cite[Theorem 6.3]{Th}, the theorem states that, for any $k \in \N$, there are $\overline{M}_k \in \N$ and a polynomial, $s_k$, of the Chern classes such that $ \overline{M}_k ( s_k(\mathrm{Ind}^G_H( \phi ))- \mathrm{Tr}^G_H (s_k ( \rho)) )=0\in H^{2k}(G;\Z)$. In lower degree, $ \overline{M}_1= 2$, $ \overline{M}_2=6$ and $s_1(\rho)=c_1(\rho)$, $s_2(\rho) = c_1(\rho)^2 -2c_2(\rho)$ are known; these mean \eqref{seq747} and \eqref{seq74799} exactly.
In particular, if $\rho$ is a one-dimensional representation $s_2(\rho)= c_1(\rho)^2$.} deduces
\begin{equation}\label{seq747}
12 c_2( \mathrm{Ind}_{H}^G( \phi) )=
6 c_1( \mathrm{Ind}_{H}^G( \phi) )^2 - 6 \mathrm{Tr}^G_H ( c_1( \phi)\smile c_1( \phi) ) \in H^4( G;\Z), \end{equation}
\begin{equation}\label{seq74799}
2 c_1( \mathrm{Ind}_{H}^G( \phi) )= 2 \mathrm{Tr}^G_H (c_1 ( \phi) )\in H^2( G;\Z).\end{equation}
Let us mention the identities $c_1(\phi)= \beta (\phi) \in H^2(H;\Z) $ from \cite[p. 68]{Th} and \begin{equation}\label{seq77}
\beta ( u \smile_{\Q/\Z} v) = u\smile_{\Q/\Z} \beta(v) , \quad \quad \mathrm{for \ any \ } u \in H^n( G;\Z ) , v\in H^m( G;\Q/\Z), \end{equation}
from \cite[V.3.(3.3)]{Br}. Since $\mathrm{Tr}_{H}^G \circ \beta = \beta \circ \mathrm{Tr}_{H}^G$ (see \cite[Exercise 4.2.4]{Ev1}), the right hand side of \eqref{seq747} is reduced to 
$6 \beta \bigl( \phi \smile_{\Q/\Z} c_1( \mathrm{Ind}_{H}^G( \phi) ) - \mathrm{Tr}^G_H ( \phi \smile_{\Q/\Z} c_1( \phi) ) \bigr).$ 
By applying this to \eqref{seq77}, we readily obtain \eqref{s23} as desired.
In addition, we immediately obtain \eqref{s256} from \eqref{seq74799} and \eqref{seq77}.
Finally, the last claim follows directly from the projection formula; see, e.g., \cite[V.(3.8)]{Br}.
\end{proof}

To summarize, by \eqref{seq771}, 
we readily reach the following conclusion:
\begin{prop}\label{prop4784}
For a representation $\rho : G \ra \mathrm{GL}_n(\C )$, we suppose \eqref{seq7712}.
Then,
\begin{equation}\label{seq77111} 12 (\beta^{-1}(c_2(\rho) )=6 \sum_{k=1}^m 
n_k(\rho) \Bigl( \bigl( \sum_{j=1}^m n_j (\rho) \mathrm{Tr}^G_{H_j}( \phi_j ) \smile \beta \circ
\mathrm{Tr}^G_{H_k}( \phi_k ) \bigr) - \mathrm{Tr}^G_{H_k} ( \phi_k \smile \beta ( \phi_k) ) \Bigr) . \end{equation}
\end{prop}
\begin{rem}\label{prop084}
Let us make a remark on a topological meaning of the first term in \eqref{seq77111}.
In general, for any $1$-cocycles $\bar{\phi}: G \ra \Q/\Z$ as in $ \mathrm{Tr}^G_{H_j} ( \phi_j ) $ and any homomorphism $f: \pi_1(M)\ra G $, we regard $\bar{\phi} \circ f $ as first cohomology classes of $M$ via $H^1(M;\Q/\Z )= \Hom ( \pi_1(M), \Q/\Z)$.
Hence, the pairing $ \langle f^* (\psi), [M] \rangle $ with $\psi =\phi_j \smile \beta( \phi_k )$ can be computed from the cohomology ring $H^*(M;\Q/\Z)$ and the Bockstein maps (i.e., the linking form) of $M$.
In addition, 
the pairing can be computed in several ways: for example, see \cite{MOO} for such computations from some quasitriangular Hopf algebras, branched coverings, or Kirby diagrams.
Meanwhile, the second term in \eqref{seq77111} will be topologically examined in Section \ref{SS3}.
\end{rem}


%
By Lemma \ref{prop1}, it is reasonable to 
describe 
explicit formula of $\beta$ and $ \mathrm{Tr}^G_H $
in the non-homogeneous complex of $G$.
First, from the definition of the Bockstein map, $ \beta(\phi):H \times H \ra \Z $ is represented by
\begin{equation}\label{t0099}\beta(\phi) ( g,h) := \left\{
\begin{array}{ll}
1 ,& \mathrm{if} \quad \overline{ \phi(g)} + \overline{ \phi(h)} \geq 1, \\
0 ,& \mathrm{if} \quad \overline{ \phi(g)} + \overline{ \phi(h)} <1.
\end{array}\right. \end{equation}
Here, for $r \in \Q/\Z $, we uniquely choose the representative $\overline{r} \in \Q \cap [0,1 )$. 
Next, let us choose a complete set of representatives of the left cosets $G/H$, $T:= \{ g_i \}_{i=1}^{|G/H|} \subset G$, such that $T \cap H=\{1_G\}.$
For $\sigma \in G$, we fix notation
\begin{equation}\label{t009} \bar{\sigma}:= H \sigma \cap T, \quad \quad \widetilde{\sigma}:= \sigma (\bar{\sigma}) \in H. \notag \end{equation}
Then,
according to \cite[Exercise II.10.2]{Br}, $ \mathrm{Tr}^G_H (\psi ): G \ra \Q/\Z$ is represented by a correspondence
\begin{equation}\label{t008} \sigma \longmapsto \sum_{i: 1 \leq i \leq |G/H| } \phi (g_i \sigma \overline{g_i \sigma}^{-1} )/|G| . \end{equation}
In addition, for a 3-cocycle $\psi : H^3 \ra A$, the transfer $\Tr_H^G(\psi )$ is known to be represented as a map $G^3 \ra A$ defined by setting
\begin{equation}\label{t007}(\sigma_1, \sigma_2, \sigma_3) \longmapsto \sum_{i: 1 \leq i \leq |G/H| } \psi ( \widetilde{\bar{g_i} \sigma_1}, \widetilde{\overline{g_i \sigma_1 } \sigma_2},
\widetilde{\overline{g_i \sigma_1\sigma_2 } \sigma_3} ). \end{equation}
In summary, combing the above formulae \eqref{t0099}--\eqref{t007} with Lemma \ref{prop1}, we can concretely obtain a representative 3-cocycle of $12 \beta^{-1}( c_2(\Ind_H^G \phi ))$ as a map $G^3 \ra \Q/\Z$.
In particular, if every 4-cocycle $C \in H^{4}(G;\Z)$ admits a representation $\rho: G \ra \mathrm{GL}_n(\C)$ such that $12 C=12 c_2(\rho)$, then we can obtain a representative $G^3 \ra \Q/\Z$ of twelvefold every 3-cocycle of $G$.
For example, if $G$ has a periodic cohomology, $G$ satisfies this condition;
see \cite{Th} for the detail and other examples, in comparison with studies of the Chern subring $\mathrm{CH}^*(G) \subset H^{\rm even}(G;\Z)$,
where $\mathrm{CH}^*(G)$ is defined to be the subring generated by the Chern classes of all complex representations of $G$.


\section{Finite groups of type $C_{m}$ and covering spaces }\label{SS3}
In this section, we will introduce a class of finite groups (Definition \ref{ex461}), and show (Theorem \ref{ex37341}) that,
if $G$ lies in the class, some multiples of the Dijkgraaf-Witten invariant of $M$ can be computed from
the cohomology rings of some finite covering spaces of $M$.
To this end, let us prepare the following lemmas:
\begin{lem}\label{ex00373491}
For a surjective homomorphism $f: \pi_1(M) \ra G$ and a subgroup $j: H \subset G$,
let $\pi : M_{f, H } \ra M$ be the finite covering associated with the subgroup $f^{-1}(H) \subset \pi_1(M)$.
Let $m$ be the greatest common divisor of $|H|$ and $|G/H|$, i.e., $m= \mathrm{g.c.d}( |H|, |G/H|)$.
Suppose $ \bar{\psi} \in H^3( G;\Q/\Z)$ and $ \psi \in H^3( H;\Q/\Z)$ such that $j^* (\bar{\psi} )= m \psi $.
Then,
\begin{equation}\label{s2454} m^2 \langle \mathrm{Tr}^G_{H} (\psi ) , f_* \circ \iota_* ([M]) \rangle = m^2 \langle \psi , \mathrm{res} \circ \iota_* (f)_*([M_{f,H}]) \rangle \in \Q/\Z .
\end{equation}
Here, the fundamental 3-class $[M_{f,H}]$ is canonically defined from $[M]$ via the covering $M_{f,H}\ra M$.
In particular, if $\psi =\phi_k \smile \beta ( \phi_k) $ as in the second term in \eqref{seq77111}, then the pairing \eqref{s2454} is computed from the cup product of the space $M_{f,H}$.
\end{lem}
\begin{proof}
Choose $s \in \N$ such that $ |G/H|= ms $ and $(s, |H|)=1$. Since $H^*( H;\Q/\Z) $ is annihilated by $|H|$, 
$s$ is invertible in computing the pairing on the cohomology of $H$. Since $ \pi_*([M_{f,H} ] )=|G/H|[M]$ by the definition of $[M_{f,H}]$, we have 
\begin{equation}\label{s24433}m^2 \langle \mathrm{Tr}^G_{H} ( \psi) ,f_* \circ \iota_* ([M])\rangle =
s^{-1}
\langle \mathrm{Tr}^G_{H} ( m \psi) ,f_* \circ \iota_* (|G/H| [M])\rangle =
s^{-1}
\langle \mathrm{Tr}^G_{H} (j^*( \bar{\psi})) ,f_* \circ \pi _* \circ \iota_* ([M_{f,H }])\rangle 
\end{equation}
Let $\mathrm{res}(f)$ be the restriction of $f$ on $\pi_1(M_{f,H})$.
By $f \circ \pi_* = j_* \circ \mathrm{res}(f)$ 
and the familiar identity $ \mathrm{Tr}^G_{H} (j^*( \bar{\psi})) =|G/H| \bar{\psi} $, \eqref{s24433} is reduced to
\begin{equation}\label{s2464}
m \langle \bar{\psi} ,j_* \circ \mathrm{res}(f)_* \circ \iota_* ([M_{f,H }]) \rangle=m \langle j^*(\bar{\psi}) , \mathrm{res}(f)_* \circ \iota_* ([M_{f,H }]) \rangle \notag \end{equation}
\begin{equation}\label{s2464}
= m^2 \langle \psi , \mathrm{res}(f)_* \circ \iota_* ([M_{f,H }]) \rangle \in \Q/\Z. \end{equation}
This computation means \eqref{s2454} exactly.
\end{proof}
\noindent We remark that the surjectivity of $f$ is not so important; indeed, if $f$ is not surjective, we may replace $H \subset G$ as $ H \cap \mathrm{Im}(f) 
\subset G \cap \mathrm{Im}(f)$.

We also give a sufficient condition for the assumption in Lemma \ref{ex00373491}:
\begin{lem}\label{ex3734199} 
Suppose an inclusion $j: H \hookrightarrow G$, and let $m \in \N$ be $ \mathrm{g.c.d.}( |H|,|G/H|)$ as above. 
Then, for any $\psi \in H^{n}(H;\Q/\Z)$, there exists $\bar{\psi} \in H^{n}(G ;\Q/\Z)$ such that $j^*( \bar{\psi }) - m \psi \in \Ker(\mathrm{Tr}^G_{H} ) $. In particular, if $\mathrm{Tr}^G_{H} $ is injective, $j^*( \bar{\psi }) = m \psi $.
\end{lem}
\begin{proof} Choose $s \in \N$ such that $ |G/H|= ms $ and $(s, |H|)=1$.
Since $H^{n}(H;\Q/\Z) $ is annihilated by $|H|$, we can define $\bar{\psi }$ to be $s^{-1} \mathrm{Tr}_{H}^G ( \psi) $. 
By the familiar equality $ |G/H|x =\mathrm{Tr}_{H}^G (j^* ( x)) $ for any $x \in H^{n}(G;\Q/\Z)$,
we have $ \mathrm{Tr}_{H}^G(j^*( \bar{\psi }) - m \psi )=0$ as required.
\end{proof}
In general, $\mathrm{Tr}^G_{H} $ is far from being injective; to solve non-injectivity, it is sensible to consider a class of finite groups satisfying the condition in Lemma \ref{ex00373491} as follows:
\begin{defn}\label{ex461}
For $m \in \Z$, we say a finite group $G$ to be {\it of type $C_{m}$}, if there are finitely many 
subgroups $j_i: H_i \subset G$ with one--dimensional representation
$\phi_i:H_i \ra \mathrm{GL}_1(\C)$ such that
\begin{enumerate}[(i)]
\item
The subring of $H^*(G;\Z)$ generated by the total Chern classes $c(\mathrm{Ind}_{H_j}^G( \phi_i) )$ with all $ i$ includes
the $m$--multiplication of the forth cohomology $m H^{4}(G;\Z)$. 
\item For any $i$, g.c.d($|H_i|$, $|G/H_i|$) is divisible by $m$, and there is $\kappa_i \in H^3(G;\Z)$ such that
$ j_i^*( \kappa_i)=m \phi_i \smile_{\Q/\Z} \beta (\phi_i)$.
\end{enumerate}
\end{defn}
Before giving such examples and their properties, we state Theorem \ref{ex37341} below.
If $\psi = \phi_k \smile \beta ( \phi_k) $, the left hand side of \eqref{s2454} is equal to the $m^2$-times of the second term in \eqref{seq77111}.
As mentioned in Remark \ref{prop084},
the right hand side of \eqref{s2454} can be computed from
the linking form of the covering space $M_{f,H_i} $.
If $G$ is of type $C_{m}$, the $m$--multiple of any cohomology 3-class $\psi$ is equal to a certain sum of the forms $\beta^{-1}( c_2 (\mathrm{Ind}_{H_i}^G( \phi_i) )) $, and satisfies the assumption of Lemma \ref{ex00373491}. 
In conclusion, we have a summary:
\begin{thm}\label{ex37341} Let $G$ be a finite group of type $C_{m}$.
Let $m' \in \mathbb{N}$ be $\mathrm{l.c.d.}(12,m^2)$. Then,
for any cohomology 3-class $\lambda \in H^3(G;\Q/\Z)$, the DW invariant \eqref{seq75559} with $\psi = m' \lambda $ can be computed from
the set $ \Hom(\pi_1(M),G)$ and the linking forms of the covering spaces $M_{f,H_i } $, where $f$ runs over $ \Hom(\pi_1(M),G) $.
\end{thm}
\begin{rem}\label{s2274} The assumption of type $C_{m}$ seems necessary because of unnaturality of the transfer.
To be precise, for any homomorphism $f:K \ra G$ and a subgroup $H \subset G$ of finite index,
$\mathrm{res}(f )^* \circ \Tr_{H}^G $ and $\Tr_{f^{-1}(H)}^K \circ f^*$ are not always equal. For instance,
for any odd prime $p \in \Z$,
if $\Z_p =:H =K\subset G:= \mathfrak{S}_p $ and $f$ is the inclusion $K=\Z_p \hookrightarrow \mathfrak{S}_p $, then $f^* \circ \Tr_{H}^G =0$ and $\Tr_{f^{-1}(H)}^K \circ f^*$ is the identity on $H^1( \bullet ;\Z_p).$
Thus, in general, \eqref{s2454} does not always hold without the assumption.
\end{rem}
We later show (Example \ref{ex3412244}) that $\mathrm{SL}_2(\F_q)$ in some cases is of type $C_2$. Thus,
\begin{cor}\label{ex373413} Let $G$ be $\mathrm{SL}_2(\F_q)$ 
with odd prime power $q \geq 10$.
Take a generator $\phi $ of $ H^3(G;\Q/\Z) \cong \Z/(q^2 -1) \Z $.
Then, 
the DW invariant \eqref{seq75559} with $\psi = 12 \phi$ can be computed from
the set $ \Hom(\pi_1(M),G)$ and the linking forms of some finite covering spaces of $M$.
\end{cor}
We now give some properties of groups of $C_m$-type, and such examples.

\begin{exa}[Cyclic group]\label{ex341} For any $n \in \Z$, every cyclic group $\Z/n$ is of type $C_{1}$ by definition.
In fact we may let $H_i$ be $G$, since $H^3(G;\Z) \cong \Z/|G|$ is generated by $ \phi \smile_{\Q/\Z} \beta (\phi)$,
where $\phi : G \ra \Q/\Z$ is an inclusion map. 
\end{exa}
\begin{exa}[Dihedral group]\label{s22384}
Suppose integers $a,n \in \N$ such that $(a,n)=1.$
Considering a surjective homomorphism $\Z/a \ra (\Z/n)^{\times}$,
we obtain the semi-direct product $G:=\Z/n \rtimes \Z/a $.
By transfer, we can easily show that $H^4(G;\Z) \neq 0$ if and only if $a=2$.
Thus, we may suppose $a=2$, that is,
$G$ is the dihedral group, $D_n$, of order $2n$.

Let $H \subset G$ be the subgroup $\Z/n \times \{ 0\}$.
As is known (see, e.g., page 74 of \cite{Th}), 
$ 2H^4( G;\Z ) \cong \Z/n $ is generated by $c_2(\rho)$, where $\rho= \Ind_{\Z_m}^{D_n} \phi $.
Namely, the group $D_n$ is of type $C_{2}$. 
Here $\phi: \Z/n \ra \mathrm{GL}_1(\C)=\Q/\Z$ is a 1-cocycle that sends $k$ to $k/n.$
\end{exa}
\begin{exa}[quaternion group]\label{ex34221}
Fix $n \in \N$.
Consider the generalized quaternion group $Q_{4n}$ of order $4n$, which has a presentation
$ \langle x,y \mid x^n =y^2 , xyx=y\rangle$. Let $H$ be the cyclic group generated by $x$.
It is known that
\[ H_1(Q_{4n} ;\Z) \cong \Z/4 , \quad H_2(Q_{4n} ;\Z) \cong 0, \quad H_3(Q_{4n} ;\Z) \cong \Z/ 4n,\]
and the inclusion $H \subset Q_{4n}$ induces an injection $\Z/2n=H_3(H) \hookrightarrow H_3(Q_{4n} ). $
Thus, $Q_{4n}$ satisfies (ii).
Hence, $ Q_{4n} $ is of type $C_{2}. $
\end{exa}
Next, to observe Example \ref{ex3412244}, we prepare a lemma, which is immediately shown by definition.
\begin{lem}\label{ex34122}
If the map $H^*(G;\Q/\Z) \ra \oplus_{K \subset G : \textrm{Sylow group}} H^*(K ;\Q/\Z) $ induced by inclusions is surjective,
and any Sylow group of $G$ is of $C_{m} $-type,
then so is $G$.
\end{lem}
\begin{exa}\label{ex3412244}
Using Lemma \ref{ex34122}, we now show that, for any odd prime power $ q \in \Z$ with $q>10$, the special linear group $\mathrm{SL}_2(\F_q)$ is of type $C_{2}$.
In fact, it is classically known (see, e.g., \cite{Hut}) that,
the inclusion from any $p$-Sylow group, $G_p$ induces an injection $H_3(G_p;\Z) \ra H_3(\mathrm{SL}_2(\F_q);\Z) \cong \Z/q^2-1$,
and that if $p$ is odd, then $G_p$ is cyclic; 
the $2$-Sylow group $G_2$ is the generalized quaternion group of order $2^t$ for some $t \in \N$, and contains a cyclic subgroup $\Z_{ 2^{t-1}}$ which is of type $C_2$ as in Example \ref{ex34221}.
In conclusion, by $H_3(H_i ;\Z ) \cong H^4(H_i;\Z) $ as in \eqref{lklk}, every sylow subgroup of $ \mathrm{SL}_2(\F_q)$ is of type $C_{2}$ as desired; hence, $\mathrm{SL}_2(\F_q)$ is of type $C_{2}$ by Lemma \ref{ex34122}.

Similarly, we can show that the projective one $\mathrm{PSL}_2(\F_q)$ is also of type $C_{4}$.
\end{exa}
\begin{rem}\label{ex34244}
More generally, we can show that every finite group with periodic group cohomology is of type $C_4$.
Here, the cohomology is said to be periodic if there is $k \in \N $ such that $H^{*+k}(G;\Z) \cong H^{*}(G;\Z) $.
The proof is as follows. It is known as Suzuki-Zassenhaus theorem (see \cite[IV.6]{AM} that such groups are completely classified as some semi-direct products or $\Z/2$-central extensions of the above groups. Thus, we can determine the type $C_4$ using similar discussions in the above examples.
\end{rem}


\begin{exa}\label{s2284}
Finally, let us give a simple computation of DW invariants from the dihedral group.
Let $n \in \mathbb{N}$ be odd, and
$G$ be the dihedral group $D_n:=\Z/n \rtimes \Z/2 $.
Hence, by \eqref{seq77111} and \eqref{s2454}, the $12$--multiple of the DW invariant \eqref{seq75559} with $\psi = \beta^{-1}(c_2(\rho) ) $ can be computed from the cohomology rings of some double covering spaces of $M$.

For more concrete computations, let us fix some 3-manifolds.
Since $ Q_{4n}$ is well known to be a subgroup of $SU(2)=S^3 $, we have a closed 3-manifold $M_n:=S^3/ Q_{4n} $.
By the Lyndon-Hochschild-Serre spectral sequence, the projection $Q_{4n} \ra D_n$ induces a surjection $H_3(Q_{4n};\Z) \ra 2 H_3(D_{n};\Z) \cong \Z/n . $
Meanwhile, the fourfold covering space of $M_n$ by the abelianization is the lens space $L(n,q)$ for some $q \in \Z/n $. 
Therefore, the covering $L(n,q) \ra M_n$ induces a canonical injection \[\mathcal{B} : \Z/n \cong \Hom ( \pi_1(L(n,q)) , \Z/n ) \hookrightarrow \Hom ( \pi_1(M_n) , D_n ) ,\]
such that the complement $\Hom ( \pi_1(M_n) , D_n ) \setminus \mathrm{Im}(\mathcal{B}) $ consists of the trivial map.
Since $\mathrm{DW}_{\beta^{-1} (c_1(\phi)^2)} ( L(n,q)) =  \sum_{j=1}^n 1_{\Z}(q j^2/n ) \in \Z[\Q/\Z]$ is known (see, e.g, \cite{DW,MOO}),
by \eqref{s2454}, we conclude $\mathrm{DW}_{\beta^{-1} c_2(\rho )} ( M_n ) = \sum_{j=0}^n 1_{\Z}( 2q j^2/n) \in \Z[\Q/\Z]$.
\end{exa}
More generally, it is sensible to examine more examples of groups that are (not) of type $C_{2} $,
and to describe the second Chern classes with computations of the associated Dijkgraaf-Witten invariants.
For example, if $G= \mathrm{SL}_2 (\mathbb{F}_q)$ as in Corollary \ref{ex373413},
we can compute the invariants via the procedure in Remark \ref{prop084}; however, the resulting are complicated.
In this paper, we omit to add other examples of the computation; in fact, the resulting computation of even Seifert homology spheres with $G= \mathrm{PSL}_2 (\mathbb{F}_p)$ is hard to write and seems complicate; see \cite{Y} for the details (see also \cite[Proposition 5.3]{Nos} for the complicate computation in nilpotent cases).




\section{Other cohomological approaches to second Chern classes}\label{SS4}
We end this paper by suggesting other three approaches to second Chern classes,
and examining renormalization of the 12--multiplication in Proposition \ref{prop4784}.

First, let us review \cite[Theorem 4]{Ev2}. Take a subgroup $H \subset G$ with a representation $\phi: H \ra \mathrm{GL}_1(\C)$.
Let $ \pi=\Ind_H^G(1)$. Then, the theorem says the equalities
\begin{equation}\label{t9921} c_1( \Ind_H^G(\phi)) =\Tr_H^G( c_1(\phi)) +c_1(\pi) \in H^2(G;\Z), 
\end{equation}
\begin{equation}\label{t991} c_2( \Ind_H^G(\phi)) = \mathcal{N}_H^G(1+c_1( \phi)) + c_1(\pi) \smile \Tr_H^G( c_1(\phi))+ c_2(\pi) \in H^4(G;\Z).
\end{equation}
Here, $ \mathcal{N}_H^G$ is the multiplicative transfer of Evens (see \cite[\S 6.1]{Ev1}), and
$2c_1(\pi)=0$ and $12 c_2( \pi)=0$ are known (compared with \eqref{seq747}).

As mentioned in \cite{Q} as an unpublished note, a representative 4-cocycle of the first term in \eqref{t991} is described as follows.
Let $ \mathfrak{S}_{|G/H|} $ be the symmetric group of order $ |G/H| $ with a permutation action on the product $H^{|G/H|}$.
We can define the semidirect product $H^{|G/H|} \rtimes \mathfrak{S}_{|G/H|} $.
Take a group homomorphism $\Phi: G \ra H^{|G/H|} \rtimes \mathfrak{S}_{|G/H|} $ defined in \cite[\S 5.2]{Ev1}, which is called {\it the monomial representation}, and consider a non-homogeneous 4-cocycle
$\tilde{c}_4 :(H^{|G/H|} \rtimes \mathfrak{S}_{|G/H|} )^4\ra \Z $ in the non-homogeneous cochain group defined by setting 
\[ [(h_1, \dots, h_{|G/H|} , \sigma)|(h_1', \dots, h_{|G/H|}' , \sigma') | (h_1'', \dots, h_{|G/H|}'' , \sigma'')| (h_1''', \dots, h_{|G/H|} ''', \sigma''')] \mapsto \]
\begin{equation}\label{t2}\sum_{ 1 \leq i \neq j \leq |G/H|} c_1(\phi)(h_{\sigma^{-1}(i)} , h_{(\sigma \sigma')^{-1}(i)}' ) \cdot c_1(\phi)(h_{(\sigma\sigma' \sigma'')^{-1}(j)}'' , h_{(\sigma\sigma'\sigma'' \sigma''')^{-1} (j) }'''),
\end{equation}
where the 2-cocycle $c_1(\phi) = \beta(\phi)$ is represented by a map $H \times H \ra \Z$ as in \eqref{t0099}. Then, $ \mathcal{N}_H^G(1+c_1( \phi)) = \Phi^*( \tilde{c}_4 )$ is shown \cite{Q}.
Moreover, in general, for a non-homogeneous 4-cocycle $C_4: G^4 \ra \Z$, the inverse $\beta^{-1}(C_4) $ is represented by a map
\begin{equation}\label{tu3}G \times G \times G \lra \Q/\Z; (g_1,g_2,g_3) \longmapsto \sum_{g \in G } \frac{C_4(g, g_1,g_2,g_3) }{|G|^2 },
\end{equation}
where we should notice that the sum $\sum_{g \in G } C_4(g, g_1,g_2,g_3 ) $ is divisible by $|G|$ for any $(g_1,g_2,g_3) \in G^3. $
This formula \eqref{tu3} of $\beta^{-1}(C_4)$ can be easily proved from the definition of the Bockstein map.
In conclusion, by \eqref{t991}--\eqref{tu3}, we can find a representative 3-cocycle of $\beta^{-1}( c_2( \Ind_H^G(\phi))) $ as a map $G^3 \ra \Q/\Z$.
Meanwhile, following an outline in \cite[\S 5]{Nos}, for some 3-manifolds $M$, we can sometimes describe concretely a chain 3-class of $[M]$ and the induced classifying map $\iota_*$; hence,
for any homomorphism $f: \pi_1(M) \ra G$,
we can find a 3-chain of the pushforward $f_*([M]) \in H_3(G;\Z)$. In conclusion,
by combing the 3-chain with the presentation of $c_2( \Ind_H^G(\phi)) $, we can compute the Dijkgraaf-Witten invariant with $\psi= \beta^{-1} c_2( \Ind_H^G(\phi))$.

Incidentally, as pointed out in \cite{Ev2, Q}, the 12--multiplication of Proposition \ref{prop4784} is essentially derived from the (stable) third homology $H_3(\mathfrak{S}_n;\Z ) \cong \Z_{12} \oplus (\Z_2)^{d_n }$ with relation to the stable homotopy group of spheres $\pi_3^S \cong \Z_{24}$, where $d_4=d_5=1,$ and $d_n=2$ for $n \geq 6$. The third framed bordism group is also isomorphic to $\pi_3^S $.
Thus, a refinement of the 12--multiple DW invariant may be interpreted from some obstructions of the sphere spectrum or some (co)-homology operations.

As the second approach, let us observe a topological approach by the splitting principle.
Given a representation $\rho: G \ra \mathrm{GL}_n(\C)$, we can construct a cellular map $\bar{\rho}: BG \ra \mathrm{Gr}(n)$ such that
$c_i(\rho) =\bar{\rho}^*(c_i)$,
where $\mathrm{Gr}(n)$ is the infinite-dimensional complex Grassmannian manifold and
$c_i \in H^{2i}( \mathrm{Gr}(n);\Z ) $ is the Chern class; see, e.g., \cite[\S 5]{Th} for the construction.
Let $\gamma $ be the tautological vector bundle on $\mathrm{Gr}(n) $ such that $c_i=c_i(\gamma)$.
For a homomorphism $f: \pi_1(M) \ra G$, the splitting principle admits a continuous map $p_{f}: M_f \ra M$ such that
$ p_f^*: H^* (M; \Z)\ra H^* (M_f; \Z)$ is injective and
the pullback $ p_f ^* \circ f^* ( \gamma )$ is decomposed as a Whiteny sum of some line bundles on $M_f$.
As in \eqref{s256}, the cohomology 3-class $\beta^{-1} p_f^*\circ f^* ( c_2(\rho) )$ can be computed from the cohomology ring of $M_f$.
In summary, if we can construct $M_f$ as a manifold and determine the cohomology ring $H^*(M_f;\Z/m)$ for any $m \in \N$,
we can compute the pairing $ \langle f^* (\beta^{-1}(c_2(\rho) )), \iota_*[M] \rangle \in \Q/\Z$.


Finally, we briefly mention modular representations. Choose distinct primes $p,\ell \in \N$.
While we considered only complex representations, we can also obtain from any representation $\rho: G \ra \GL_n(\overline{\mathbb{F}}_p)$ uniquely a Chern class $c_i (\rho) \in H^{2i }(G; \widehat{\Z}_{\ell} )$, where $\widehat{\Z}_{\ell} $ is the $\ell$-adic integer ring; see \cite{Kr} or \cite[p.111]{Th} for the details.
Therefore, similarly to \eqref{lklk}, we have a cohomology 3-class $ \beta^{-1} (c_2(\rho))$.
If $p, \ell >3$, and a similar Riemann-Roch theorem is true, we can easily obtain similar results in Sections \ref{SS2} and \ref{SS3}.
For example, the cohomology $H^*(\GL_n ( \F_{p^k}); \Z[1/p])$ is generated by such Chern classes, as is stated in the famous result of Quillen \cite{Qu},
while $ \GL_n ( \F_{p^k})$ is not of type $C_{2,m}$ for enough large $n,k$.
Thus, 
to analyze the DW invariant, we might employ Chern classes from some modular representations of $G$.

\appendix

\section{Relation to the Chern-Simons invariant}\label{SS8884}
In this appendix, we explain the viewpoint of the DW invariant in terms of the Chern-Simons invariant.
As mentioned in \cite{CS}, the Chern-Simons 3-class can be considered as a group 3-cocycle $\mathrm{CCS}$ in $C^3(\mathrm{GL}_N(\mathbb{C}) ; \C/\Z)$ with $N \geq 2$.
Thus,
given a homomorphism $F: \pi_1(M) \ra \mathrm{GL}_N(\mathbb{C})$,
{\it the Chern-Simons invariant}, $\mathrm{CS}(F)$, is defined as $\langle \mathrm{CCS} ,F_* \iota_* ([M]) \rangle \in \C/\Z$.

\begin{prop}\label{s22848} Take a representation $\rho: G \ra \mathrm{GL}_N(\mathbb{C}) $. 
Then, the image of $ \rho^*(\mathrm{CCS} )$ is contained in $\Q/\Z$, and
\begin{equation}\label{888} \rho^*(\mathrm{CCS} ) = \beta^{-1}(c_2(\rho )) \in H^3(G;\Q/\Z).
\end{equation}
In particular, the DW invariant with $\psi=\beta^{-1}(c_2(\rho )) $ is equal to
$ \sum_{f }1_{\Z} \mathrm{CS}( \rho \circ f ) \in \Z[\Q/\Z]$
where $f$ runs over $\Hom( \pi_1(M) , G )$.
\end{prop}
There are many procedures to compute the Chern-Simons invariant.
For example, \cite{JW} mentions relations from the $e$-invariant and $\eta$-invariant in differential topology.
If the image of $F$ is contained in $\mathrm{SL}_2(\mathbb{C} )$, the cocycle expression of $ 2 \mathrm{CCS} $ is described in \cite{DZ}.
As a special case, if $M$ is a Seifert manifold over $S^2$ and $H_* (M;\Z) \cong H_* (S^3;\Z)$ and the Seifertor\footnote{As is known in low-dimensional topology, the center of $\pi_1(M)$ contains $\Z$ and is called {\it the Seifertor.}} is sent to trivial by $f $,
Theorem C in \cite{JW}\footnote{In the theorem C, the image is required to be contained in $\mathrm{SL}_N(\C)$; however,
if we consider the inclusion $\mathrm{GL}_{N}(\C) \hookrightarrow \mathrm{SL}_{N+1}(\C)$ which sends $A$ to $\begin{pmatrix}
A& 0\\
0 & \det(A)^{-1} \\
\end{pmatrix} $, we can apply the theorem for the computation of $ \mathrm{CS}(F).$} gives an explicit formula to compute the $2N$--multiple Chern-Simons invariant $2N \rho^*(\mathrm{CCS} ) $; In conclusion, following Proposition \ref{s22848}, the DW invariant with $\psi=2N \beta^{-1}(c_2(\rho )) $ of
Seifert homology spheres can be sometimes easily computed with the help of a computer program; see also \cite{Ch, Y} for the computation of $\Hom(\pi_1(M), G)$.

Now, we give the proof of Proposition \ref{s22848}.
Let $K$ be a complex Lie group of finite connected components,
and $I_k(K, \C)$ denote the ring consisting of $K$-invariant polynomials on the Lie algebra $\mathfrak{k}$.
Recall from the classical Chern-Weil theory that there is a natural
homomorphism $W: I_k(K, \C) \ra H^{2k}(BK^{\rm top}, \C)$, where $BK^{\rm top}$ is the topological classifying space of $K$.
Let $r$ denote the map $H^*(BK^{\rm top} ;\Z) \ra H^* (BK^{\rm top}; \C)$ induced by the inclusion $\Z \hookrightarrow \C$.
From \cite{CS}, consider
\[ J^k (K , \C):=\{(P,u) \in I^k(K, \C) \times H^{2k}(BK^{\rm top};\Z) \mid W(P)=r(u)\}. \]
Then, in \cite{CS}, there is a homomorphism $\textrm{C-C-S}_K: J^k (K , \C ) \ra H^{2k-1} (BK ; \C/\Z)$ such that
any homomorphism $\phi: G \ra K$ ensures the naturality $ \phi^* \circ \textrm{C-C-S}_G = \textrm{C-C-S}_K \circ \phi^*
$; see also Page 1 of \cite{DZ}.
As important examples, it is widely known that if $K= \mathrm{GL}_n (\C)$ and $ u=c_2 \in H^{4}(BK^{\rm top}, \Z) $ is the Chern class, and $P(A) =((\mathrm{Tr}A)^2- \mathrm{Tr}(A^2) )/ 4 \pi^2 $, then $ \textrm{C-C-S}_K (P,u)= \textrm{CCS}$.
Furthermore, if $G$ is of finite order and $ u=c_2(\rho)$, then $BG^{\rm top} =BG $ and $ I^k(G, \C) $ is zero; thus, $ \textrm{C-C-S}_G$ is equal to $\beta_{\C/\Z}^{-1}(c_2(\rho )) $
by definition, where $\beta_{\C/\Z}$ is the Bockstein map from
\begin{equation}\label{seq379}
0 \lra \Z \lra\C \lra \C/\Z \lra 0 \quad \mathrm{(exact)}. \end{equation}
\begin{proof}[Proof of Proposition \ref{s22848}] As mentioned above, $\beta_{\C/\Z}^{-1}(c_2(\rho )) $ lies in $H^3(G ;\Q/\Z)$.
Since $\beta^{-1} = \beta_{\C/\Z}^{-1}$,
the above naturality deduces to 
\[ \beta^{-1} (c_2(\rho )) = \beta_{\C/\Z}^{-1} (c_2(\rho )) = \textrm{C-C-S}_G(\rho^*(P), \rho^*(c_2))= \rho^* (\textrm{C-C-S}_K(P, c_2) ) = \rho^* (\textrm{CCS}) ,\]
which is the required \eqref{888}.
\end{proof}

\subsection*{Acknowledgments}
The work was partially supported by JSPS KAKENHI, Grant Number 00646903.
The author sincerely thanks an anonymous referee for carefully reading this paper and making many comments.

\vskip 1pc

\normalsize

\noindent
Department of Mathematics, Tokyo Institute of Technology
2-12-1
Ookayama, Meguro-ku Tokyo 152-8551 Japan

\end{document}